\documentclass[11pt]{amsart}

\theoremstyle{plain}
\newtheorem{theorem}                {Theorem}      [section]
\newtheorem*{theorem1}                {Theorem \ref{thm:pmc_red_1}}
\newtheorem*{theorem2}                {Theorem \ref{thm:pmc_red_2}}
\newtheorem*{theorem3}                {Theorem \ref{thm:surface_red1}}
\newtheorem*{theorem4}                {Theorem \ref{thm:surface_red2}}

\newtheorem{proposition}  [theorem]  {Proposition}
\newtheorem{corollary}    [theorem]  {Corollary}
\newtheorem{lemma}        [theorem]  {Lemma}

\theoremstyle{definition}

\newtheorem{remark}         {Remark}[section]
\newtheorem{definition}   [theorem]  {Definition}

\DeclareMathOperator{\trace}{trace}

 \DeclareMathOperator{\id}{I}
 
\DeclareMathOperator{\ric}{Ric}

\DeclareMathOperator{\Span}{span}
 
\DeclareMathOperator{\cst}{constant}

\numberwithin{equation}{section}

\begin{document}

\title[Complete pmc submanifolds]
{On complete submanifolds with parallel mean curvature in product spaces}

\author{Dorel~Fetcu}
\author{Harold~Rosenberg}

\address{Department of Mathematics\\
"Gh. Asachi" Technical University of Iasi\\
Bd. Carol I no. 11 \\
700506 Iasi, Romania} \email{dfetcu@math.tuiasi.ro; dorel@impa.br}

\address{IMPA\\ Estrada Dona Castorina\\ 110, 22460-320 Rio de
Janeiro, Brasil} \email{rosen@impa.br}

\thanks{The first author was partially supported by a grant of the Romanian National Authority for
Scientific Research, CNCS -- UEFISCDI, project number
PN-II-RU-TE-2011-3-0108.}

\begin{abstract} We prove a Simons type formula for submanifolds with parallel mean curvature vector field in product spaces of type $M^n(c)\times\mathbb{R}$, where $M^n(c)$ is a space form with constant sectional curvature $c$, and then we use it to characterize some of these submanifolds.
\end{abstract}



\maketitle

\section{Introduction}

In 1968, James Simons obtained an equation for the Laplacian of the
second fundamental form of a minimal subamanifold of a Riemannian
manifold (see \cite{JS}).  He then applied this theorem in several
ways; in particular by characterizing certain minimal submanifolds
of spheres. Over the years, such formulas, nowadays called Simons
type equations, proved to be a powerful tool not only for studying
minimal submanifolds in Riemannian manifolds, but also, more
generaly, for studying submanifolds with constant mean curvature
(cmc submanifolds) or with parallel mean curvature vector (pmc
submanifolds). A special attention was paid to cmc and pmc
submanifolds in space forms, articles like
\cite{AdC,AT,CN,E,NS,WS,Y} being only a few examples of
contributions on this topic in which Simons type formulas are used
to prove gap and reduction of codimension theorems. An excellent
presentation of the classical result of Simons and some of its
applications can be found in the very recent book \cite{CM}. The
authors point out, for example, how Simons' equation can be used to
obtain curvature bounds for minimal surfaces with small total
curvature and also curvature estimates for stable minimal surfaces
in $\mathbb{R}^3$, and then, more generally, for stable minimal
hypersurfaces in $\mathbb{R}^n$.

Recently, such equations were obtained for cmc and pmc submanifolds in product spaces of type $M^n(c)\times\mathbb{R}$, where $M^n(c)$ stands for an $n$-dimensional space form with
constant sectional curvature $c$, and then used to characterize some of these submanifolds (see, for example, \cite{B,FOR}). More exactly, in \cite{B} the author computed the Laplacian of the second fundamental form of a cmc surface in $M^3(c)\times\mathbb{R}$, as well as the Laplacian of the traceless part of the Abresch-Rosenberg differential introduced in \cite{AR} for such surfaces, whilst in \cite{FOR} it was found the expression of the Laplacian of $|A_H|^2$ for a pmc submanifold in $M^n(c)\times\mathbb{R}$ with shape operator $A$ and mean curvature vector field $H$.

In our paper, we first compute the Laplacian of the second
fundamental form of a pmc submanifold in $M^n(c)\times\mathbb{R}$
and then we use this Simons type formula to prove some gap theorems
for pmc submanifolds in $M^n(c)\times\mathbb{R}$ when $c>0$ and the
mean curvature vector field $H$ of the submanifold makes a constant
angle with the unit vector field $\xi$ tangent to $\mathbb{R}$, or
when $c<0$ and $H$ is orthogonal to $\xi$.

Our main results are the following four theorems.

\begin{theorem1} Let $\Sigma^m$ be an immersed complete non-minimal pmc submanifold in
$M^n(c)\times\mathbb{R}$, $n>m\geq 3$, $c>0$, with mean curvature vector field $H$ and second fundamental form $\sigma$. If the angle between $H$ and $\xi$ is constant and
$$
|\sigma|^2+\frac{2c(2m+1)}{m}|T|^2\leq 2c+\frac{m^2}{m-1}|H|^2,
$$
where $T$ is the tangent part of $\xi$, then $\Sigma^m$ is a totally umbilical cmc hypersurface in
$M^{m+1}(c)$.
\end{theorem1}

\begin{theorem2} Let $\Sigma^m$ be an immersed complete non-minimal pmc submanifold in
$M^n(c)\times\mathbb{R}$, $n>m\geq 3$, $c<0$, with mean curvature vector field $H$ and second fundamental form $\sigma$. If $H$ is orthogonal to $\xi$  and
$$
|\sigma|^2+\frac{2c(m+1)}{m}|T|^2\leq 4c+\frac{m^2}{m-1}|H|^2,
$$
then $\Sigma^m$ is a totally umbilical cmc hypersurface in
$M^{m+1}(c)$.
\end{theorem2}

\begin{theorem3} Let $\Sigma^2$ be a complete non-minimal pmc surface in $M^n(c)\times\mathbb{R}$, $n>2$, $c>0$, such that
the angle between $H$ and $\xi$ is constant and
$$
|\sigma|^2+3c|T|^2\leq 4|H|^2+2c.
$$
Then, either
\begin{enumerate}
\item $\Sigma^2$ is pseudo-umbilical and lies in $M^n(c)$$;$ or

\item $\Sigma^2$ is a torus $\mathbb{S}^1(r)\times\mathbb{S}^1\Big(\sqrt{\frac{1}{c}-r^2}\Big)$ in $M^3(c)$, with $r^2\neq\frac{1}{2c}$.
\end{enumerate}
\end{theorem3}

\begin{theorem4} Let $\Sigma^2$ be a complete non-minimal pmc surface in $M^n(c)\times\mathbb{R}$, $n>2$, $c<0$, such that
$H$ is orthogonal to $\xi$ and
$$
|\sigma|^2+5c|T|^2\leq 4|H|^2+4c.
$$
Then $\Sigma^2$ is pseudo-umbilical and lies in $M^n(c)$.
\end{theorem4}

\noindent \textbf{Acknowledgments.} The first author would like to
thank the IMPA in Rio de Janeiro for providing a very stimulative work environment during the preparation of this paper.

\section{Preliminaries}

Let $M^n(c)$ be a space form, i.e. a simply-connected
$n$-dimensional manifold with constant sectional curvature $c$. Thus, $M^n(c)$ will be the sphere $\mathbb{S}^n(c)$, the
Euclidean space, or the hyperbolic space $\mathbb{H}^n(c)$, as
$c>0$, $c=0$, or $c<0$. Now, let us consider the product manifold
$\bar M=M^n(c)\times\mathbb{R}$. The expression of the curvature
tensor $\bar R$ of such a manifold can be obtained from
$$
\langle\bar R(X,Y)Z,W\rangle=c\{\langle d\pi Y, d\pi Z\rangle\langle d\pi X, d\pi W\rangle-\langle d\pi X, d\pi Z\rangle\langle d\pi Y, d\pi W\rangle\},
$$
where $\pi:\bar M=M^n(c)\times\mathbb{R}\rightarrow M^n(c)$ is the projection map. After a straightforward computation we get
\begin{equation}\label{eq:barR}
\begin{array}{ll}
\bar R(X,Y)Z=&c\{\langle Y, Z\rangle X-\langle X, Z\rangle Y-\langle Y,\xi\rangle\langle Z,\xi\rangle X+\langle X,\xi\rangle\langle Z,\xi\rangle Y\\ \\&+\langle X,Z\rangle\langle Y,\xi\rangle\xi-\langle Y,Z\rangle\langle X,\xi\rangle\xi\},
\end{array}
\end{equation}
where $\xi$ is the unit vector tangent to $\mathbb{R}$.

Let $\Sigma^m$ be an $m$-dimensional submanifold of $\bar M$. From
the equation of Gauss
$$
\begin{array}{ll}
\langle R(X,Y)Z,W\rangle=&\langle\bar R(X,Y)Z,W\rangle\\ \\&+\sum_{\alpha=m+1}^{n+1}\{\langle A_{\alpha}Y,Z\rangle\langle A_{\alpha}X,W\rangle-\langle A_{\alpha}X,Z\rangle\langle A_{\alpha}Y,W\rangle\},
\end{array}
$$
we obtain the expression of its curvature tensor
\begin{equation}\label{eq:R}
\begin{array}{ll}
R(X,Y)Z=&c\{\langle Y, Z\rangle X-\langle X, Z\rangle Y-\langle Y,T\rangle\langle Z,T\rangle X+\langle X,T\rangle\langle Z,T\rangle Y\\ \\&+\langle X,Z\rangle\langle Y,T\rangle T-\langle Y,Z\rangle\langle X,T\rangle T\}\\ \\&+\sum_{\alpha=m+1}^{n+1}\{\langle A_{\alpha}Y,Z\rangle A_{\alpha}X-\langle A_{\alpha}X,Z\rangle A_{\alpha}Y\},
\end{array}
\end{equation}
where $T$ is the component of $\xi$ tangent to $\Sigma^m$ and $A$ is
the shape operator defined by the equation of Weingarten
$$
\bar\nabla_XV=-A_VX+\nabla^{\perp}_XV,
$$
for any vector field $X$ tangent to $\Sigma^m$ and any normal vector
field $V$. Here $\bar\nabla$ is the Levi-Civita connection on $\bar
M$, $\nabla^{\perp}$ is the connection in the normal bundle, and
$A_{\alpha}=A_{E_{\alpha}}$, $\{E_{\alpha}\}_{\alpha=m+1}^{n+1}$
being a local orthonormal frame field in the normal bundle.

\begin{definition} A submanifold $\Sigma^m$ of $M^n(c)\times\mathbb{R}$ is called a \textit{vertical cylinder} over $\Sigma^{m-1}$ if $\Sigma^m=\pi^{-1}(\Sigma^{m-1})$, where $\pi:M^n(c)\times\mathbb{R}\rightarrow M^n(c)$ is the projection map and $\Sigma^{m-1}$ is a submanifold of $M^n(c)$.
\end{definition}

It is easy to see that vertical cylinders $\Sigma^m=\pi^{-1}(\Sigma^{m-1})$ are characterized by the fact that $\xi$ is tangent to $\Sigma^m$.

\begin{definition} If the mean curvature vector field $H$ of a submanifold $\Sigma^m$ is
parallel in the normal bundle, i.e.
$\nabla^{\perp}H=0$, then $\Sigma^m$ is called a \textit{pmc submanifold}.
\end{definition}

\begin{remark}\label{rem:cyl} It is straightforward to verify that $\Sigma^m=\pi^{-1}(\Sigma^{m-1})$
is a pmc vertical cylinder in $M^n(c)\times\mathbb{R}$ if and only
if $\Sigma^{m-1}$ is a pmc submanifold in $M^n(c)$. Moreover, the
mean curvature vector field of $\Sigma^m$ is $H=\frac{m-1}{m} H_0$,
where $H_0$ is the mean curvature vector field of $\Sigma^{m-1}$.
\end{remark}

We end this section by recalling the following three results, which
we shall use later in this paper.

\begin{lemma}[\cite{BYC}]\label{l:chen} Let $a_1,\ldots,a_m$, where $m>1$, and $b$ be real numbers such that
\begin{equation}\label{ineq_c1}
\Big(\sum_{i=1}^n a_i\Big)^2\geq (n-1)\sum_{i=1}^n a_i^2+b.
\end{equation}
Then, for all $i\neq j$, we have
\begin{equation}\label{ineq_c2}
2a_ia_j\geq\frac{b}{n-1}.
\end{equation}
Moreover, if the inequality \eqref{ineq_c1} is strict, then so are the inequalities \eqref{ineq_c2}.
\end{lemma}

\begin{lemma}[\cite{LL}]\label{lili} Let $A_1,\ldots,A_p$, where $p\geq 2$, be symmetric $m\times m$ matrices. Then
$$
\sum_{\alpha,\beta=1}^p\{N(A_{\alpha}A_{\beta}-A_{\beta}A_{\alpha})+(\trace(A_{\alpha}A_{\beta}))^2\}\leq\frac{3}{2}\Big(\sum_{\alpha=1}^pN(A_{\alpha})\Big)^2,
$$
where $N(A)=\trace (A^tA)$. Equality holds if and only if either
\begin{enumerate}
\item  $A_1=\ldots=A_p=0$$;$ or

\item only two matrices $A_{\alpha_0}$ and $A_{\beta_0}$ are different from the null $m\times m$ matrix. Moreover, in this case, $N(A_{\alpha_0})=N(A_{\beta_0})=L$ and there exists an orthogonal matrix $T$ such that
$$
T^tA_{\alpha_0}T=\sqrt{\frac{L}{2}}\left(\begin{array}{ccccc}1&0&0&\ldots&0\\0&-1&0&\ldots&0\\0&0&0&\ldots&0\\\vdots&\vdots&\vdots&\vdots&\vdots\\0&0&0&\ldots&0\end{array}\right),\quad T^tA_{\beta_0}T=\sqrt{\frac{L}{2}}\left(\begin{array}{ccccc}0&1&0&\ldots&0\\1&0&0&\ldots&0\\0&0&0&\ldots&0\\\vdots&\vdots&\vdots&\vdots&\vdots\\0&0&0&\ldots&0\end{array}\right)
$$

\end{enumerate}
\end{lemma}

\begin{theorem}[Omori-Yau Maximum Principle, \cite{Y2}]\label{OY} If $\Sigma^m$ is a complete Riemannian
manifold with Ricci curvature bounded from below, then for any
smooth function $u\in C^2(\Sigma^m)$ with $\sup_{\Sigma^m}
u<+\infty$ there exists a sequence of points
$\{p_k\}_{k\in\mathbb{N}}\subset \Sigma^m$ satisfying
$$
\lim_{k\rightarrow\infty}u(p_k)=\sup_{\Sigma^m}u,\quad |\nabla u|(p_k)<\frac{1}{k}\quad\textnormal{and}\quad\Delta u(p_k)<\frac{1}{k}.
$$

\end{theorem}

\section{A Simons type formula for pmc submanifolds in $M^n(c)\times\mathbb{R}$}\label{section_S}

Let $\Sigma^m$ be an $m$-dimensional pmc submanifold in
$M^n(c)\times\mathbb{R}$ with mean curvature vector field $H$.

In the following, we shall compute the Laplacian of the squared norm
of the second fundamental form $\sigma$ of $\Sigma^m$, where
$\sigma$ is defined by the equation of Gauss
$$
\bar\nabla_XY=\nabla_XY+\sigma(X,Y)
$$
for any tangent vector fields $X$ and $Y$.

Let $\{E_{m+1},\ldots,E_{n+1}\}$ be a local orthonormal frame field
in the normal bundle. Then, normal connection forms
$s_{\alpha\beta}$ are determined by
$$
\nabla^{\perp}_XE_{\alpha}=\sum_{\beta=m+1}^{n+1}s_{\alpha\beta}(X)E_{\beta}
$$
for any vector field $X$ tangent to $\Sigma^n$ and any
$\alpha\in\{m+1,\ldots,n+1\}$. It is easy to see that
$s_{\alpha\beta}=-s_{\beta\alpha}$ and that
$$
\begin{array}{ll}
\nabla^{\perp}_XH&=\frac{1}{m}\nabla^{\perp}_X(\trace\sigma)=\frac{1}{m}\sum_{\alpha=m+1}^{n+1}\nabla^{\perp}_X((\trace A_{\alpha})E_{\alpha})
\\ \\&=\frac{1}{m}\sum_{\alpha=m+1}^{n+1}(X(\trace A_{\alpha})-\sum_{\beta=m+1}^{n+1}s_{\alpha\beta}(X)\trace A_{\beta})E_{\alpha}.
\end{array}
$$
Therefore, the mean curvature vector field $H$ is parallel if and
only if
\begin{equation}\label{eq:H}
X(\trace A_{\alpha})-\sum_{\beta=m+1}^{n+1}s_{\alpha\beta}(X)\trace A_{\beta}=0
\end{equation}
for all $\alpha$'s.

Now, from the Codazzi equation,
$$
\begin{array}{cl}
\langle \bar R(X,Y)Z,E_{\alpha}\rangle=&\langle\nabla^{\perp}_X\sigma(Y,Z),E_{\alpha}\rangle-\langle\sigma(\nabla_XY,Z),E_{\alpha}\rangle
-\langle\sigma(Y,\nabla_XZ),E_{\alpha}\rangle\\ \\&-\langle\nabla^{\perp}_Y\sigma(X,Z),E_{\alpha}\rangle+\langle\sigma(\nabla_YX,Z),E_{\alpha}\rangle
+\langle\sigma(X,\nabla_YZ),E_{\alpha}\rangle,
\end{array}
$$
we get
$$
\begin{array}{rl}
\langle \bar R(X,Y)Z,E_{\alpha}\rangle=&X(\langle A_{\alpha}Y,Z\rangle)-\langle\sigma(Y,Z),\nabla^{\perp}_XE_{\alpha}\rangle
-\langle A_{\alpha}(\nabla_XY),Z\rangle\\ \\&-\langle A_{\alpha}Y,\nabla_XZ\rangle-Y(\langle A_{\alpha}X,Z\rangle)+\langle\sigma(X,Z),\nabla^{\perp}_YE_{\alpha}\rangle
\\ \\&+\langle A_{\alpha}(\nabla_YX),Z\rangle+\langle A_{\alpha}X,\nabla_YZ\rangle
\\ \\=&\langle(\nabla_XA_{\alpha})Y-(\nabla_YA_{\alpha})X,Z\rangle\\ \\&-\langle\sum_{\beta=m+1}^{n+1}(s_{\alpha\beta}(X)A_{\beta}Y-s_{\alpha\beta}(Y)A_{\beta}X),Z\rangle.
\end{array}
$$
Therefore, using \eqref{eq:barR}, we obtain
\begin{equation}\label{eq:Codazzi}
\begin{array}{rl}
(\nabla_XA_{\alpha})Y=&(\nabla_YA_{\alpha})X+\sum_{\beta=m+1}^{n+1}(s_{\alpha\beta}(X)A_{\beta}Y-s_{\alpha\beta}(Y)A_{\beta}X)\\ \\&+c\langle E_{\alpha},N\rangle(\langle Y,T\rangle X-\langle X,T\rangle Y),
\end{array}
\end{equation}
where $N$ is the normal part of $\xi$.

Next, we have the following Weitzenb\"ock fromula
\begin{equation}\label{eq:Laplacian}
\frac{1}{2}\Delta|A_{\alpha}|^2=|\nabla A_\alpha|^2+\langle\trace\nabla^2A_{\alpha},A_{\alpha}\rangle,
\end{equation}
where we extended the metric $\langle,\rangle$ to the tensor space in the standard way.

The second term in the right hand side of \eqref{eq:Laplacian} can
be calculated by using a method introduced in \cite{NS} and developed in \cite{E}.

Let us consider
$$
C_{\alpha}(X,Y)=(\nabla^2 A_{\alpha})(X,Y)=\nabla_X(\nabla_YA_{\alpha})-\nabla_{\nabla_XY}A_{\alpha},
$$
and note that we have the following Ricci commutation
formula
\begin{equation}\label{eq:C} C_{\alpha}(X,Y)=C_{\alpha}(Y,X)+[R(X,Y),A_{\alpha}].
\end{equation}

Next, consider an orthonormal basis $\{e_i\}_{i=1}^{m}$ in
$T_p\Sigma^m$, $p\in\Sigma^m$, extend $e_i$ to vector fields $E_i$
in a neighborhood of $p$ such that $\{E_i\}$ is a geodesic frame
field around $p$, and let us denote $X=E_k$. We have
$$
(\trace\nabla^2A_{\alpha})X=\sum_{i=1}^mC_{\alpha}(E_i,E_i)X.
$$

Using equation \eqref{eq:Codazzi}, we get, at $p$,
$$
\begin{array}{rl}
C_{\alpha}(E_i,X)E_i=&\nabla_{E_i}((\nabla_{X}A_{\alpha})E_i)\\ \\=&\nabla_{E_i}((\nabla_{E_i}A_{\alpha})X)+\nabla_{E_i}(\sum_{\beta=m+1}^{n+1}(s_{\alpha\beta}(X)A_{\beta}E_i-s_{\alpha\beta}(E_i)A_{\beta}X))\\ \\&+c\nabla_{E_i}(\langle E_{\alpha},N\rangle(\langle E_i,T\rangle X-\langle X,T\rangle E_i))
\end{array}
$$
and then
\begin{equation}\label{eq:1}
\begin{array}{lcl}
C_{\alpha}(E_i,X)E_i&=&C_{\alpha}(E_i,E_i)X\\ \\&&+\nabla_{E_i}(\sum_{\beta=m+1}^{n+1}(s_{\alpha\beta}(X)A_{\beta}E_i-s_{\alpha\beta}(E_i)A_{\beta}X))\\ \\&&+c\langle\sum_{\beta=m+1}^{n+1}s_{\alpha\beta}(E_i)E_\beta,N\rangle(\langle E_i,T\rangle X-\langle X,T\rangle E_i)\\ \\&&-c\langle A_{\alpha}E_i,T\rangle(\langle E_i,T\rangle X-\langle X,T\rangle E_i)\\ \\&&+c\langle E_{\alpha},N\rangle(\langle A_NE_i,E_i\rangle X-\langle A_NX,E_i\rangle E_i),
\end{array}
\end{equation}
where we used $\sigma(E_i,T)=-\nabla^{\perp}_{E_i}N$ and
$\nabla_{E_i}T=A_NE_i$, which follow from the fact that $\xi$ is
parallel, i.e. $\bar\nabla\xi=0$.

We also have, at $p$,
\begin{equation}\label{eq:2}
C_{\alpha}(X,E_i)E_i=\nabla_{X}((\nabla_{E_i}A_{\alpha})E_i),
\end{equation}
and then, from \eqref{eq:C}, \eqref{eq:1} and \eqref{eq:2}, we get
$$
\begin{array}{lcl}
C_{\alpha}(E_i,E_i)X&=&\nabla_{X}((\nabla_{E_i}A_{\alpha})E_i)+[R(E_i,X),A_{\alpha}]E_i\\ \\&&-\nabla_{E_i}(\sum_{\beta=m+1}^{n+1}(s_{\alpha\beta}(X)A_{\beta}E_i-s_{\alpha\beta}(E_i)A_{\beta}X))\\ \\&&-c\langle\sum_{\beta=m+1}^{n+1}s_{\alpha\beta}(E_i)E_\beta,N\rangle(\langle E_i,T\rangle X-\langle X,T\rangle E_i)\\ \\&&+c\langle A_{\alpha}E_i,T\rangle(\langle E_i,T\rangle X-\langle X,T\rangle E_i)\\ \\&&-c\langle E_{\alpha},N\rangle(\langle A_NE_i,E_i\rangle X-\langle A_NX,E_i\rangle E_i).
\end{array}
$$
Since $\nabla_{E_i}A_{\alpha}$ is symmetric, from \eqref{eq:Codazzi}
one obtains
\begin{equation}\label{eq:sum_nabla}
\begin{array}{lcl}
\langle\sum_{i=1}^m(\nabla_{E_i}A_{\alpha})E_i,Z\rangle&=&\sum_{i=1}^m\langle E_i,(\nabla_{E_i}A_{\alpha})Z\rangle=\sum_{i=1}^m\langle E_i,(\nabla_{Z}A_{\alpha})E_i\rangle\\ \\&&-\sum_{i=1}^m\langle\sum_{\beta=m+1}^{n+1}s_{\alpha\beta}(Z) A_{\beta}E_i,E_i\rangle\\ \\&&+\sum_{i=1}^m\langle\sum_{\beta=m+1}^{n+1}s_{\alpha\beta}(E_i)A_{\beta}E_i,Z\rangle \\ \\&&+c\langle E_{\alpha},N\rangle\sum_{i=1}^m\langle E_i,\langle Z,T\rangle E_i-\langle E_i,T\rangle Z\rangle,
\end{array}
\end{equation}
which, together with \eqref{eq:H}, leads to
\begin{equation}\label{eq:sum}
\begin{array}{rl}
\langle\sum_{i=1}^m(\nabla_{E_i}A_{\alpha})E_i,Z\rangle=&
Z(\trace A_{\alpha})-\sum_{\beta=m+1}^{n+1}s_{\alpha\beta}(Z)\trace A_{\beta}\\ \\&+\langle\sum_{i=1}^m\sum_{\beta=m+1}^{n+1}s_{\alpha\beta}(E_i)A_{\beta}E_i+c(m-1)\langle E_{\alpha},N\rangle T,Z\rangle\\ \\=&\langle\sum_{i=1}^m\sum_{\beta=m+1}^{n+1}s_{\alpha\beta}(E_i)A_{\beta}E_i+c(m-1)\langle E_{\alpha},N\rangle T,Z\rangle,
\end{array}
\end{equation}
for any vector $Z$ tangent to $\Sigma^m$.

Therefore, we have
\begin{equation}\label{eq:trace}
\begin{array}{lcl}
(\trace\nabla^2A_{\alpha})X&=&\sum_{i=1}^m C_{\alpha}(E_i,E_i)X\\ \\&=&\sum_{i=1}^m\sum_{\beta=m+1}^{n+1}\{X(s_{\alpha\beta}(E_i))A_{\beta}E_i+s_{\alpha\beta}(E_i)\nabla_XA_{\beta}E_i\\ \\&&-E_i(s_{\alpha\beta}(X))A_{\beta}E_i-s_{\alpha\beta}(X)\nabla_{E_i}A_{\beta}E_i \\ \\&&+E_i(s_{\alpha\beta}(E_i))A_{\beta}X+s_{\alpha\beta}(E_i)\nabla_{E_i}A_{\beta}X\}\\ \\&&+c(m-1)\langle\sum_{\beta=m+1}^{n+1}s_{\alpha\beta}(X)E_{\beta},N\rangle T\\ \\&&-c\sum_{i=1}^m\langle\sum_{\beta=m+1}^{n+1}s_{\alpha\beta}(E_i)E_\beta,N\rangle(\langle E_i,T\rangle X-\langle X,T\rangle E_i)\\ \\&&+c\langle A_{\alpha}T,T\rangle X-c\langle X,T\rangle A_{\alpha}T-cm\langle E_{\alpha}, N\rangle\langle H,N\rangle X\\ \\&&+cm\langle E_{\alpha},N\rangle A_NX-c(m-1)\langle A_{\alpha}T,X\rangle T\\ \\&&+\sum_{i=1}^m[R(E_i,X),A_{\alpha}]E_i.
\end{array}
\end{equation}

Now, using the Ricci equation
$$
\langle R^{\perp}(X,Y)E_{\alpha},E_{\beta}\rangle=\langle[A_{\alpha},A_{\beta}]X,Y\rangle+\langle\bar R(X,Y)E_{\alpha},E_{\beta}\rangle,
$$
we get, after a straightforward computation,
\begin{equation}\label{eq:a}
\begin{array}{l}
\sum_{i=1}^m\sum_{\beta=m+1}^{n+1}(X(s_{\alpha\beta}(E_i))A_{\beta}E_i-E_i(s_{\alpha\beta}(X))A_{\beta}E_i)\\ \\=\sum_{i=1}^m\sum_{\beta=m+1}^{n+1}((\nabla_Xs_{\alpha\beta})(E_i)A_{\beta}E_i-(\nabla_{E_i}s_{\alpha\beta})(X)A_{\beta}E_i)\\ \\
=\sum_{\beta=m+1}^{n+1}A_{\beta}[A_{\alpha},A_{\beta}]X-\sum_{i=1}^m\sum_{\beta,\gamma=m+1}^{n+1}s_{\alpha\gamma}(E_i)s_{\gamma\beta}(X)A_{\beta}E_i\\ \\ \ \ \ +\sum_{i=1}^m\sum_{\beta,\gamma=m+1}^{n+1}s_{\alpha\gamma}(X)s_{\gamma\beta}(E_i)A_{\beta}E_i.
\end{array}
\end{equation}

From \eqref{eq:Codazzi}, we have
\begin{equation}\label{eq:b}
\begin{array}{lcl}
\sum_{i=1}^m\sum_{\beta=m+1}^{n+1}s_{\alpha\beta}(E_i)\nabla_XA_{\beta}E_i&=&\sum_{i=1}^m\sum_{\beta=m+1}^{n+1}s_{\alpha\beta}(E_i)(\nabla_XA_{\beta})E_i\quad
\\ \\
&=&\sum_{i=1}^m\sum_{\beta=m+1}^{n+1}s_{\alpha\beta}(E_i)\{(\nabla_{E_i}A_{\beta})X
\\ \\&&-c\langle E_{\beta},N\rangle(\langle X,T\rangle E_i-\langle E_i,T\rangle X)
\\ \\&&-\sum_{\gamma=m+1}^{n+1}(s_{\beta\gamma}(E_i)A_{\gamma}X-s_{\beta\gamma}(X)A_{\gamma}E_i)\}.
\end{array}
\end{equation}

We use now \eqref{eq:sum} to compute
\begin{equation}\label{eq:c}
\begin{array}{lcl}
\sum_{i=1}^m\sum_{\beta=m+1}^{n+1}s_{\alpha\beta}(X)\nabla_{E_i}A_{\beta}E_i&=&\sum_{i=1}^m\sum_{\beta=m+1}^{n+1}s_{\alpha\beta}(X)(\nabla_{E_i}A_{\beta})E_i
\\ \\&=&\sum_{i=1}^m\sum_{\beta,\gamma=m+1}^{n+1}s_{\alpha\beta}(X)s_{\beta\gamma}(E_i)A_{\gamma}E_i
\\ \\&& +c(m-1)\langle\sum_{\beta=m+1}^{n+1}s_{\alpha\beta}(X)E_{\beta},N\rangle T.
\end{array}
\end{equation}

From the Gauss equation \eqref{eq:R} of $\Sigma^m$, we get
\begin{equation}\label{eq:d}
\begin{array}{lcl}
\sum_{i=1}^mR(E_i,X)A_{\alpha}E_i&=&c\{A_{\alpha}X-(\trace A_{\alpha})X+(\trace A_{\alpha})\langle X,T\rangle
T\\ \\&&-\langle A_{\alpha}X,T\rangle T
-\langle X,T\rangle A_{\alpha}T+\langle A_{\alpha}T,T\rangle X\}
\\ \\&&+\sum_{\beta=m+1}^{n+1}\{A_{\beta}A_{\alpha}A_{\beta} X-(\trace(A_{\alpha}A_{\beta}))A_{\beta}X\},
\end{array}
\end{equation}
and
\begin{equation}\label{eq:e}
\begin{array}{lcl}
\sum_{i=1}^m A_{\alpha}R(E_i,X)E_i&=&-c\{(m-1-|T|^2)A_{\alpha}X-(m-2)\langle X,T\rangle A_{\alpha}T\}\\ \\
&&+\sum_{\beta=m+1}^{n+1}\{A_{\alpha}A_{\beta}^2X-(\trace A_{\beta})A_{\alpha}A_{\beta}X\}.
\end{array}
\end{equation}

Finally, taking into account that
$$
E_i(s_{\alpha\beta}(E_i))A_{\beta}X=(\nabla_{E_i}s_{\alpha\beta})(E_i)A_{\beta}X
$$
and then replacing \eqref{eq:a}, \eqref{eq:b}, \eqref{eq:c},
\eqref{eq:d} and \eqref{eq:e} in \eqref{eq:trace}, we obtain, after
a long but straightforward computation,
\begin{equation}\label{quarter}
\begin{array}{lcl}
\langle\trace\nabla^2A_{\alpha},A_{\alpha}\rangle&=&\sum_{i=1}^m\langle(\trace\nabla^2A_{\alpha})E_i,A_{\alpha}E_i\rangle\\ \\&=&\sum_{i=1}^m\{\sum_{\beta=m+1}^{n+1}2s_{\alpha\beta}(E_i)\trace((\nabla_{E_i}A_{\beta})A_{\alpha})\\ \\&&-\sum_{\beta,\gamma=m+1}^{n+1}s_{\alpha\gamma}(E_i)s_{\gamma\beta}(E_i)\trace(A_{\alpha}A_{\beta})\\ \\&&+\sum_{\beta=m+1}^{n+1}(\nabla_{E_i}s_{\alpha\beta})(E_i)\trace(A_{\alpha}A_{\beta})\} \\ \\&&+c\{(m-|T|^2)|A_{\alpha}|^2-2m|A_{\alpha}T|^2+3(\trace A_{\alpha})\langle A_{\alpha}T,T\rangle\\ \\&&+m(\trace(A_NA_{\alpha}))\langle E_{\alpha},N\rangle-(\trace A_{\alpha})^2\\ \\&&-m(\trace A_{\alpha})\langle H,N\rangle\langle E_{\alpha},N\rangle\}\\ \\&&+\sum_{\beta=m+1}^{n+1}\{(\trace A_{\beta})(\trace(A_{\alpha}^2A_{\beta}))+\trace[A_{\alpha},A_{\beta}]^2\\ \\&&-(\trace(A_{\alpha}A_{\beta}))^2\}.
\end{array}
\end{equation}

From equation \eqref{eq:Laplacian}, we know that
\begin{equation}\label{eq:semifinal}
\frac{1}{2}\Delta|\sigma|^2=\frac{1}{2}\sum_{\alpha=m+1}^{n+1}\Delta|A_{\alpha}|^2=\sum_{\alpha=m+1}^{n+1}\{|\nabla A_{\alpha}|^2+\langle\trace\nabla^2A_{\alpha},A_{\alpha}\rangle\},
\end{equation}
and, in order to estimate this Laplacian, we first note that
$$
\sum_{\alpha=m+1}^{n+1}(\trace A_{\alpha})\langle A_{\alpha}T,T\rangle=m\langle\sigma(T,T),H\rangle,
\sum_{\alpha=m+1}^{n+1}(\trace(A_NA_{\alpha}))\langle E_{\alpha},N\rangle=|A_N|^2,
$$
$$
\sum_{\alpha=m+1}^{n+1}(\trace A_{\alpha})\langle H,N\rangle\langle E_{\alpha},N\rangle=m\langle H,N\rangle^2,\quad \sum_{\alpha=m+1}^{n+1}(\trace A_{\alpha})^2=m^2|H|^2,
$$
and, since $s_{\alpha\beta}=-s_{\beta\alpha}$, that
$$
\sum_{\alpha,\beta=m+1}^{n+1}(\nabla_{E_i}s_{\alpha\beta})(E_i)\trace(A_{\alpha}A_{\beta})=0.
$$

Next, we easily get
$$
\begin{array}{lcl}
(\nabla^{\perp}\sigma)(X,Y,Z)&=&\nabla^{\perp}_X\sigma(Y,Z)-\sigma(\nabla_XY,Z)-\sigma(Y,\nabla_XZ)\\ \\
&=&\sum_{\alpha=m+1}^{n+1}\langle(\nabla_XA_{\alpha})Y-\sum_{\beta=m+1}^{n+1}s_{\alpha\beta}(X)A_{\beta}Y,Z\rangle E_{\alpha}
\end{array}
$$
for all tangent vector fields $X$, $Y$ and $Z$, and then
$$
\begin{array}{lcl}
|\nabla^{\perp}\sigma|^2&=&\sum_{i,j,k=1}^m|(\nabla^{\perp}\sigma)(E_i,E_j,E_k)|^2\\ \\
&=&\sum_{\alpha=m+1}^{n+1}\sum_{i,j=1}^m\langle (\nabla_{E_i}A_{\alpha})E_j-\sum_{\beta=m+1}^{n+1}s_{\alpha\beta}(E_i)A_{\beta}E_j,\\ \\
&& \ \ \ \ \ \ \ \ \ \ \ \ \ \ \ \ \ \ \ \ \ \ \ (\nabla_{E_i}A_{\alpha})E_j-\sum_{\gamma=m+1}^{n+1}s_{\alpha\gamma}(E_i)A_{\gamma}E_j\rangle,
\end{array}
$$
which means that
$$
\begin{array}{lcl}
|\nabla^{\perp}\sigma|^2
&=&\sum_{\alpha={m+1}}^{n+1}\{|\nabla A_{\alpha}|^2+\sum_{i=1}^m(\sum_{\beta=m+1}^{n+1}2s_{\alpha\beta}(E_i)\trace((\nabla_{E_i}A_{\beta})A_{\alpha})\\ \\&&-\sum_{\beta,\gamma=m+1}^{n+1}s_{\alpha\gamma}(E_i)s_{\gamma\beta}(E_i)\trace(A_{\alpha}A_{\beta}))\}.
\end{array}
$$

Using \eqref{quarter} and \eqref{eq:semifinal}, we can state the following proposition.

\begin{proposition}\label{p:delta} Let $\Sigma^m$ be a pmc submanifold of $M^n(c)\times\mathbb{R}$, with mean curvature vector field $H$, shape operator $A$, and second fundamental form $\sigma$. Then we have
$$
\begin{array}{lcl}
\frac{1}{2}\Delta|\sigma|^2&=&|\nabla^{\perp}\sigma|^2+c\{(m-|T|^2)|\sigma|^2-2m\sum_{\alpha=m+1}^{n+1}|A_{\alpha} T|^2\\ \\&&+3m\langle \sigma(T,T),H\rangle+m|A_N|^2-m^2\langle H,N\rangle^2-m^2|H|^2\}\\ \\&&+\sum_{\alpha,\beta=m+1}^{n+1}\{(\trace A_{\beta})(\trace(A_{\alpha}^2A_{\beta}))+\trace[A_{\alpha},A_{\beta}]^2\\ \\&&-(\trace(A_{\alpha}A_{\beta}))^2\},
\end{array}
$$
where $\{E_{\alpha}\}_{\alpha=m+1}^{n+1}$ is a local orthonormal frame field in the normal bundle.
\end{proposition}

\begin{corollary}\label{coro_min} If $\Sigma^m$ is a minimal submanifold of $M^n(c)\times\mathbb{R}$, then we have
$$
\begin{array}{lcl}
\frac{1}{2}\Delta|\sigma|^2&=&|\nabla^{\perp}\sigma|^2+c\{(m-|T|^2)|\sigma|^2-2m\sum_{\alpha=m+1}^{n+1}|A_{\alpha} T|^2+m|A_N|^2\}\\ \\&&+\sum_{\alpha,\beta=m+1}^{n+1}\{\trace[A_{\alpha},A_{\beta}]^2-(\trace(A_{\alpha}A_{\beta}))^2\}.
\end{array}
$$
\end{corollary}

Now, let us consider $\Sigma^m$ a non-minimal pmc submanifold of
$M^n(c)\times\mathbb{R}$ and then, for any normal vector field $V$,
define $\phi_V=A_V-\frac{\trace A_V}{m}\id$, the traceless part of
$A_V$. We shall also consider $\phi$ the traceless part of $\sigma$,
given by
$$
\phi(X,Y)=\sigma(X,Y)-\langle X,Y\rangle H.
$$
It is easy to see that $\nabla^{\perp}\phi=\nabla^{\perp}\sigma$, $|\sigma|^2=|\phi|^2+m|H|^2$ and $|A_V|^2=|\phi_V|^2+\frac{(\trace A_V)^2}{m}$. It is also easy to obtain, from the Ricci equation, that if a normal vector field $V$ is parallel in the normal bundle, then $[A_V,A_U]=0$ for all normal vector fields $U$.

Let $\{E_{m+1},\ldots,E_{n+1}\}$ be a local orthonormal frame field in the normal bundle such that $E_{m+1}=\frac{H}{|H|}$. Then, we obtain the following corollary directly from Proposition \ref{p:delta}.

\begin{corollary} If $\Sigma^m$ is a non-minimal pmc submanifold of $M^n(c)\times\mathbb{R}$, then we have
$$
\begin{array}{lcl}
\frac{1}{2}\Delta|\phi|^2&=&|\nabla^{\perp}\phi|^2+(c(m-|T|^2)+m|H|^2)|\phi|^2-2cm\sum_{\alpha=m+1}^{n+1}|\phi_{\alpha} T|^2\\ \\&&-cm\langle \phi(T,T),H\rangle+cm|\phi_N|^2+m|H|\sum_{\alpha=m+1}^{n+1}\trace(\phi_{\alpha}^2\phi_{m+1})\\ \\&&+\sum_{\alpha,\beta>m+1}\trace[\phi_{\alpha},\phi_{\beta}]^2-\sum_{\alpha,\beta=m+1}^{n+1}(\trace(\phi_{\alpha}\phi_{\beta}))^2.
\end{array}
$$
\end{corollary}

In the following, we shall compute the Laplacian of the squared norm
of the tangent part $T$ of $\xi$.

As above, let us consider an orthonormal basis $\{e_i\}_{i=1}^{m}$ in
$T_p\Sigma^m$, $p\in\Sigma^m$, and then extend $e_i$ to vector fields $E_i$
in a neighborhood of $p$ such that
$\{E_i\}$ is a geodesic frame field around $p$.
Then, at $p$, we have
$$
\begin{array}{lcl}
\frac{1}{2}\Delta|T|^2&=&\sum_{i=1}^m(\langle\nabla_{E_i}T,\nabla_{E_i}T\rangle+\langle\nabla_{E_i}\nabla_{E_i}T,T\rangle)\\ \\
&=&|A_N|^2+\sum_{i=1}^m\langle\nabla_{E_i}A_NE_i,T\rangle
\end{array}
$$
and, since $\nabla_XA_N$ is symmetric,
$$
\begin{array}{lcl}
\sum_{i=1}^m\langle\nabla_{E_i}A_NE_i,T\rangle&=&\sum_{i=1}^m\langle(\nabla_{E_i}A_N)E_i,T\rangle\\ \\
&=&\sum_{i=1}^m\langle(\nabla_{E_i}A_N)T,E_i\rangle=\sum_{i=1}^m\langle\nabla_{E_i}A_NT
-A_N\nabla_{E_i}T,E_i\rangle\\ \\&=&\sum_{i=1}^m\langle\nabla_{E_i}\nabla_TT
-\nabla_{\nabla_{E_i}T}T,E_i\rangle\\ \\&=&\sum_{i=1}^m\langle\nabla_{E_i}\nabla_TT
+\nabla_{[T,E_i]}T,E_i\rangle\\ \\&=&\sum_{i=1}^m(\langle\nabla_T\nabla_{E_i}T,E_i\rangle
-\langle R(T,E_i)T,E_i\rangle)\\ \\&=&\sum_{i=1}^m(\langle\nabla_TA_NE_i,E_i\rangle
-\langle R(T,E_i)T,E_i\rangle)\\ \\&=&T(\trace A_N)
-\sum_{i=1}^m\langle R(T,E_i)T,E_i\rangle\\ \\&=&mT(\langle H,N\rangle)
-\sum_{i=1}^m\langle R(T,E_i)T,E_i\rangle\\ \\&=&-m\langle \sigma(T,T),H\rangle)
-\sum_{i=1}^m\langle R(T,E_i)T,E_i\rangle
\end{array}
$$
where we used $\nabla_XT=A_NX$ and $\nabla^{\perp}_XN=-\sigma(X,T)$.

From the Gauss equation \eqref{eq:R}, it follows that
$$
\sum_{i=1}^m\langle R(T,E_i)T,E_i\rangle=c(1-m)|T|^2(1-|T|^2)+\sum_{\alpha=m+1}^{n+1}\{|A_{\alpha}T|^2-(\trace A_{\alpha})\langle A_{\alpha}T,T\rangle\},
$$
and then we get
$$
\begin{array}{lcl}
\frac{1}{2}\Delta|T|^2&=&|A_N|^2-m\langle\sigma(T,T),H\rangle)+c(m-1)|T|^2(1-|T|^2)
\\ \\&&-\sum_{\alpha=m+1}^{n+1}\{|A_{\alpha}T|^2-(\trace A_{\alpha})\langle A_{\alpha}T,T\rangle\},
\end{array}
$$
where $\{E_{\alpha}\}_{\alpha=m+1}^{n+1}$ is a local orthonormal
frame field in the normal bundle.

We conclude with the following proposition.

\begin{proposition}\label{pTgen} Let $\Sigma^m$ be an $m$-dimensional pmc submanifold in $M^n(c)\times\mathbb{R}$,
with shape operator $A$. Then we have
$$
\frac{1}{2}\Delta|T|^2=|A_N|^2+c(m-1)|T|^2(1-|T|^2)-\sum_{\alpha=m+1}^{n+1}|A_{\alpha}T|^2.
$$
\end{proposition}

\section{Some gap theorems for pmc submanifolds in $M^n(c)\times\mathbb{R}$}

In this Section we shall present some applications of Propositions \ref{p:delta} and \ref{pTgen} in the study of pmc submanifolds. First we have the following result.

\begin{proposition}\label{thm:cmc2} Let $\Sigma^m$ be an immersed complete pmc submanifold in
$M^n(c)\times\mathbb{R}$ with second fundamental form $\sigma$.
If
$$
\sup_{\Sigma^m}\{|\sigma|^2+c(m-1)|T|^2\}<\max\{0,c(m-1)\},
$$
then either
\begin{enumerate}
\item $\Sigma^m$ lies in $M^n(c)$, if $c>0$$;$ or

\item $\Sigma^m$ is a vertical cylinder $\pi^{-1}(\Sigma^{m-1})$ over a pmc submanifold $\Sigma^{m-1}$ in $M^n(c)$, if $c<0$.
\end{enumerate}
\end{proposition}

\begin{proof} Let us consider first the case when $c>0$. Then, from Proposition \ref {pTgen}, using our hypothesis, we have that
$$
\begin{array}{ll}
\frac{1}{2}\Delta|T|^2&=|A_N|^2+c(m-1)|T|^2(1-|T|^2)-\sum_{\alpha=m+1}^{n+1}|A_{\alpha}T|^2\\ \\&\geq |T|^2(c(m-1)(1-|T|^2)-|\sigma|^2)\\ \\&\geq 0.
\end{array}
$$

Next, let us consider a local orthonormal frame field $\{E_i\}_{i=1}^{m}$ on $\Sigma^m$, $X$ a unit tangent vector field, and $\{E_{\alpha}\}_{\alpha=m+1}^{n+1}$ an orthonormal frame field in the normal bundle. From equation \eqref{eq:R}, we get the expression of the Ricci curvature of our submanifold
$$
\begin{array}{lll}
\ric X&=&\sum_{i=1}^m\langle R(E_i,X)X,E_i\rangle\\ \\
&=&\sum_{i=1}^m\{c(|X|^2-\langle X,E_i\rangle^2-\langle X,T\rangle^2+2\langle X,T\rangle\langle T, E_i\rangle\langle X, E_i\rangle\\ \\&&-\langle T, E_i\rangle^2|X|^2)
+\sum_{\alpha=m+1}^{n+1}(\langle A_{\alpha}E_i,E_i\rangle\langle A_{\alpha}X,X\rangle-\langle A_{\alpha}X,E_i\rangle^2)\}
\\ \\&=&
c(m-1-|T|^2-(m-2)\langle X,T\rangle^2)+m\langle A_HX,X\rangle-\sum_{\alpha=m+1}^{n+1}|A_{\alpha}X|^2.
\end{array}
$$
It follows that
$$
\begin{array}{ll}
\ric X&\geq c(m-1)(1-|T|^2)-m|A_HX|-\sum_{\alpha=m+1}^{n+1}|A_{\alpha}|^2\\ \\ &\geq-m|A_H|-|\sigma|^2.
\end{array}
$$
Since
$|\sigma|$ is bounded by hypothesis, we can see that the Ricci curvature is bounded
from below, and then the Omori-Yau Maximum Principle holds on $\Sigma^m$.

Therefore, we can use Theorem \ref{OY} with $u=|T|^2$.
It follows that there exists a sequence of points
$\{p_k\}_{k\in\mathbb{N}}\subset \Sigma^m$ satisfying
$$
\lim_{k\rightarrow\infty}|T|^2(p_k)=\sup_{\Sigma^m}|T|^2\quad\textnormal{and}\quad\Delta|T|^2(p_k)<\frac{1}{k}.
$$
Since $\sup_{\Sigma^m}\{|\sigma|^2+c(m-1)|T|^2\}<c(m-1)$, it follows that
$0=\lim_{k\rightarrow\infty}|T|^2(p_k)=\sup_{\Sigma^m}|T|^2$,
which means that $T=0$, i.e. $\Sigma^m$ lies in $M^n(c)$.

When $c<0$, we come to the conclusion in the same way as above, using the facts that
$$
\begin{array}{ll}
\frac{1}{2}\Delta|N|^2&=-\frac{1}{2}\Delta|T|^2=-|A_N|^2-c(m-1)|T|^2(1-|T|^2)+\sum_{\alpha=m+1}^{n+1}|A_{\alpha}T|^2\\ \\&\geq |N|^2(-|\sigma|^2-c(m-1)|T|^2)\\ \\&\geq 0,
\end{array}
$$
and that
$$
\ric X\geq c(m-1)-m|A_H|-|\sigma|^2,
$$
and then applying Theorem \ref{OY} to function $u=|N|^2$.
\end{proof}

For minimal submanifolds in $M^n(c)\times\mathbb{R}$, with $c>0$, we have the following result.

\begin{proposition} Let $\Sigma^m$ be a complete minimal submanifold in
$M^n(c)\times\mathbb{R}$, with $c>0$.
If
$$
\sup_{\Sigma^m}\{3|\sigma|^2+2c(2m+1)|T|^2\}<2cm,
$$
then $\Sigma^m$ is a totally geodesic submanifold in $M^n(c)$.
\end{proposition}

\begin{proof} From Corollary \ref{coro_min}, since Schwarz inequality implies that $|A_{\alpha}T|^2\leq|T|^2|A_{\alpha}|^2$, using $|A_N|^2\geq 0$ and Lemma \ref{lili}, we obtain
$$
\Delta|\sigma|^2\geq -(3|\sigma|^2+2c((2m+1)|T|^2-m))|\sigma|^2\geq 0.
$$
As we have seen, since $|\sigma|$ is bounded, the Ricci curvature of
$\Sigma^m$ is bounded from below, and then we can apply the
Omori-Yau Maximum Principle to function $u=|\sigma|^2$. One
obtains that there exists a sequence of points
$\{p_k\}_{k\in\mathbb{N}}\subset \Sigma^m$ satisfying
$$
\lim_{k\rightarrow\infty}|\sigma|^2(p_k)=\sup_{\Sigma^m}|\sigma|^2\quad\textnormal{and}\quad\Delta|\sigma|^2(p_k)<\frac{1}{k},
$$
from where it follows that
$0=\lim_{k\rightarrow\infty}|\sigma|^2(p_k)=\sup_{\Sigma^m}|\sigma|^2$,
which means that $\sigma=0$. Moreover, $A_N=0$ and then the hypothesis
imply that $|T|^2=\cst<1$. From Proposition \ref{pTgen}, it
follows that $T=0$, which means that our submanifold is totally geodesic
in $M^n(c)$.
\end{proof}

Before stating our first main result, we shall prove the following lemma, which shall be then used in its proof.

\begin{lemma}\label{HN}  Let $\Sigma^m$ be an immersed non-minimal pmc submanifold in
$M^n(c)\times\mathbb{R}$ with mean curvature vector field $H$. Then we have
$$
\Delta\langle H,N\rangle=-c(m-1)|T|^2\langle H,N\rangle-\trace(A_HA_N).
$$
\end{lemma}

\begin{proof} Let $\{E_i\}_{i=1}^m$ be a geodesic frame field around a point $p\in\Sigma^m$. Then, since $H$ is parallel and $\nabla^{\perp}_XN=-\sigma(X,T)$, we have, at $p$,
$$
\Delta\langle H,N\rangle=\sum_{i=1}^m E_i(E_i(\langle H,N\rangle))=-\sum_{i=1}^m E_i(\langle A_HT,E_i\rangle).
$$
Using the facts that $\nabla_{X}A_H$ is symmetric and that $\nabla_XT=A_NX$, and also equation \eqref{eq:sum_nabla}, we get
$$
\begin{array}{ll}
\Delta\langle H,N\rangle&=-\sum_{i=1}^m E_i(\langle A_HT,E_i\rangle)=-\sum_{i=1}^m\langle\nabla_{E_i}A_HT,E_i\rangle\\ \\&=-\sum_{i=1}^m(\langle(\nabla_{E_i}A_H)T,E_i\rangle+\langle A_H\nabla_{E_i}T,E_i\rangle)\\ \\&=-\sum_{i=1}^m\langle(\nabla_{E_i}A_H)E_i,T\rangle-\trace(A_HA_N)\\ \\&=-c(m-1)|T|^2\langle H,N\rangle-\trace(A_HA_N).
\end{array}
$$
\end{proof}

Our main results are similar to those obtained in \cite{AT,CN} for
the pmc submanifolds of a sphere and Euclidean space, and, again as
in the above cited papers, their proofs rely on the use of formulas
obtained in Section \ref{section_S} and of Lemmas \ref{l:chen} and
\ref{lili}.

\begin{theorem}\label{thm:pmc_red_1} Let $\Sigma^m$ be a complete non-minimal pmc submanifold in
$M^n(c)\times\mathbb{R}$, $n>m\geq 3$, $c>0$, with mean curvature vector field $H$ and second fundamental form $\sigma$. If the angle between $H$ and $\xi$ is constant and
\begin{equation}\label{cond_thm}
|\sigma|^2+\frac{2c(2m+1)}{m}|T|^2\leq 2c+\frac{m^2}{m-1}|H|^2,
\end{equation}
then $\Sigma^m$ is a totally umbilical cmc hypersurface in
$M^{m+1}(c)$.
\end{theorem}

\begin{proof} We shall prove first that $\Sigma^m$ actually lies in a space form $M^{m+1}(c)$, and, in order to do that, we will show that, if $\{E_{m+1},\ldots,E_{n+1}\}$ is a local orthonormal frame field in the normal bundle such that $E_{m+1}=\frac{H}{|H|}$, then $A_{\alpha}=0$ for all $\alpha>m+1$.

Let us recall now a formula proved in \cite{FOR}, which can be also obtained as a particular case of the computation in Section \ref{section_S}, tacking into account that, since $E_{m+1}$ is parallel, we have $[A_{m+1},A_{\alpha}]=0$ for all $\alpha\geq m+1$,
\begin{equation}\label{eq:delta_H}
\begin{array}{lcl}
\frac{1}{2}\Delta|A_{m+1}|^2&=&|\nabla A_{m+1}|^2+c\{(m-|T|^2)|A_{m+1}|^2-2m|A_{m+1}T|^2\\ \\&&+3m\langle \sigma(T,T),H\rangle+m(\trace(A_NA_{m+1}))\langle E_{m+1},N\rangle\\ \\&&-m^2\langle H,N\rangle^2-m^2|H|^2\}\\ \\&&+\sum_{\alpha=m+1}^{n+1}\{(\trace A_{\alpha})(\trace(A_{m+1}^2A_{\alpha}))-(\trace(A_{m+1}A_{\alpha}))^2\}.
\end{array}
\end{equation}

Next, we define the function $|\mathcal{A}|^2$ on $\Sigma^m$ by $|\mathcal{A}|^2=\sum_{\alpha>m+1}|A_{\alpha}|^2$,
and, using \eqref{eq:delta_H}, we obtain, from Proposition \ref{p:delta}, that
\begin{equation}\label{eq:A}
\begin{array}{lcl}
\frac{1}{2}\Delta|\mathcal{A}|^2&=&\sum_{\alpha>m+1}|\nabla^{\ast}A_{\alpha}|^2+c\{(m-|T|^2)|\mathcal{A}|^2-2m\sum_{\alpha>m+1}|A_{\alpha} T|^2\\ \\&&+m|A_N|^2-m(\trace(A_NA_{m+1}))\langle E_{m+1},N\rangle\}\\ \\&&+\sum_{\alpha>m+1}\{(\trace A_{m+1})(\trace(A_{\alpha}^2A_{m+1}))-(\trace(A_{\alpha}A_{m+1}))^2\}\\ \\&&+\sum_{\alpha,\beta>m+1}\{\trace[A_{\alpha},A_{\beta}]^2-(\trace(A_{\alpha}A_{\beta}))^2\},
\end{array}
\end{equation}
where $\nabla^{\ast}$ is the sum of the tangent and normal connections and
$$
\nabla^{\ast}_XA_{\alpha}=\nabla_XA_{\alpha}-\sum_{\beta>m+1}s_{\alpha\beta}(X)A_{\beta}.
$$

The Schwarz inequality implies that
\begin{equation}\label{estimate1}
-\sum_{\alpha>m+1}|A_{\alpha} T|^2\geq-|T|^2\sum_{\alpha>m+1}|A_{\alpha}|^2=-|T|^2|\mathcal{A}|^2.
\end{equation}

From Lemma \ref{HN}, since $\langle H,N\rangle=\cst$, we have
\begin{equation}\label{estimate2}
|A_N|^2-(\trace(A_NA_{m+1}))\langle E_{m+1},N\rangle=|A_N|^2+c(m-1)|T|^2\langle E_{m+1},N\rangle^2\geq 0.
\end{equation}

Since $\trace[A_{\alpha},A_{\beta}]^2=-N(A_{\alpha}A_{\beta}-A_{\beta}A_{\alpha})$, using Lemma \ref{lili}, we get
\begin{equation}\label{estimate3}
\sum_{\alpha,\beta>m+1}\{\trace[A_{\alpha},A_{\beta}]^2-(\trace(A_{\alpha}A_{\beta}))^2\}\geq
-\frac{3}{2}\Big(\sum_{\alpha>m+1}|A_{\alpha}|^2\Big)^2=-\frac{3}{2}|\mathcal{A}|^4.
\end{equation}

Next, we shall evaluate the term
$$
\sum_{\alpha>m+1}\{(\trace A_{m+1})(\trace(A_{\alpha}^2A_{m+1}))-(\trace(A_{\alpha}A_{m+1}))^2\}
$$
in \eqref{eq:A}. In order to do that, we note first that, since $[A_{m+1},A_{\alpha}]=0$, the matrices $A_{m+1}$ and $A_{\alpha}$ can be diagonalized simultaneously, for each $\alpha>m+1$. Let $\lambda_i$ and $\lambda_i^{\alpha}$, $i=1,\ldots,m$, be the eigenvalues of $A_{m+1}$ and $A_{\alpha}$, respectively. Then, for each $\alpha>m+1$, we have
\begin{equation}\label{eq:term}
\begin{array}{r}
(\trace A_{m+1})(\trace(A_{\alpha}^2A_{m+1}))-(\trace(A_{\alpha}A_{m+1}))^2\quad\quad\\ \\=(\sum_{i=1}^{m}\lambda_i)(\sum_{j=1}^m\lambda_j(\lambda_j^{\alpha})^2)
-(\sum_{i=1}^{m}\lambda_i\lambda_i^{\alpha})(\sum_{j=1}^m\lambda_j\lambda_j^{\alpha})\\ \\=\frac{1}{2}\sum_{i,j=1}^m\lambda_i\lambda_j(\lambda_i^{\alpha}-\lambda_j^{\alpha})^2.
\end{array}
\end{equation}

Our hypothesis \eqref{cond_thm} can be written as
$$
(m|H|)^2\geq (m-1)|A_{m+1}|^2+(m-1)\Big(|\mathcal{A}|^2+\frac{2c(2m+1)}{m}|T|^2-2c\Big)
$$
which means that
\begin{equation}\label{eq:chen0}
\Big(\sum_{i=1}^m\lambda_i\Big)^2\geq(m-1)\sum_{i=1}^m(\lambda_i)^2+(m-1)\Big(|\mathcal{A}|^2+\frac{2c(2m+1)}{m}|T|^2-2c\Big).
\end{equation}
Thus, from Lemma \ref{l:chen}, it follows that
\begin{equation}\label{eq:ij}
\lambda_i\lambda_j\geq \frac{1}{2}|\mathcal{A}|^2+\frac{c(2m+1)}{m}|T|^2-c,
\end{equation}
for $i\neq j$, and then
\begin{equation}\label{eq:intermediar_e4}
\begin{array}{lcl}
\frac{1}{2}\sum_{i,j=1}^m\lambda_i\lambda_j(\lambda_i^{\alpha}-\lambda_j^{\alpha})^2&\geq&\frac{1}{2}\Big(\frac{1}{2}|\mathcal{A}|^2+\frac{c(2m+1)}{m}|T|^2-c\Big)\sum_{i,j=1}^m(\lambda_i^{\alpha}-\lambda_j^{\alpha})^2\\ \\&=&\Big(\frac{1}{2}|\mathcal{A}|^2+\frac{c(2m+1)}{m}|T|^2-c\Big)\sum_{i,j=1}^m((\lambda_i^{\alpha})^2-\lambda_i^{\alpha}\lambda_j^{\alpha})\\ \\&=&(\frac{m}{2}|\mathcal{A}|^2+c(2m+1)|T|^2-cm)|A_{\alpha}|^2\\ \\&&-\Big(\frac{1}{2}|\mathcal{A}|^2+\frac{c(2m+1)}{m}|T|^2-c\Big)\Big(\sum_{i=1}^m\lambda_i^{\alpha}\Big)^2\\ \\&=&(\frac{m}{2}|\mathcal{A}|^2+c(2m+1)|T|^2-cm)|A_{\alpha}|^2.
\end{array}
\end{equation}

Replacing in \eqref{eq:term}, we get
\begin{equation}\label{estimate4}
\begin{array}{r}
\sum_{\alpha>m+1}\{(\trace A_{m+1})(\trace(A_{\alpha}^2A_{m+1}))-(\trace(A_{\alpha}A_{m+1}))^2\}\\ \\\geq (\frac{m}{2}|\mathcal{A}|^2+c(2m+1)|T|^2-cm)|\mathcal{A}|^2.
\end{array}
\end{equation}

Now, from \eqref{eq:A}, \eqref{estimate1}, \eqref{estimate2}, \eqref{estimate3} and \eqref{estimate4}, one obtains
\begin{equation}\label{eq:finalA}
\frac{1}{2}\Delta|\mathcal{A}|^2\geq\frac{m-3}{2}|\mathcal{A}|^4.
\end{equation}

As we have seen in Proposition \ref{thm:cmc2}, the fact that
$|\sigma|$ is bounded implies that the Ricci curvature of $\Sigma^m$
is bounded from below. Therefore we can apply Theorem \ref{OY} to
function $u=|\mathcal{A}|^2$, and we get that there exists a
sequence of points $\{p_k\}_{k\in\mathbb{N}}\subset\Sigma^m$
satisfying
$$
\lim_{k\rightarrow\infty}|\mathcal{A}|^2(p_k)=\sup_{\Sigma^m}|\mathcal{A}|^2
\quad\textnormal{and}\quad\Delta |\mathcal{A}|^2(p_k)<\frac{1}{k}.
$$

From the inequality \eqref{eq:finalA} it follows that
$$
0=\lim_{k\rightarrow\infty}\Delta |\mathcal{A}|^2(p_k)\geq(m-3)\sup_{\Sigma^m}|\mathcal{A}|^2\geq 0,
$$
i.e. $(m-3)\sup_{\Sigma^m}|\mathcal{A}|^2=0$. Therefore, we get that
$m=3$ or $|\mathcal{A}|^2=0$.

Next, we shall split our study in two cases as $m\geq 4$ or $m=3$.

\textbf{Case I: $m\geq 4$.} In this case, we have
$|\mathcal{A}|^2=0$, and then $A_{\alpha}=0$ for all $\alpha>m+1$.
Moreover, all inequalities \eqref{estimate1}, \eqref{estimate2},
\eqref{estimate3} and \eqref{estimate4} become equalities. Since
$A_N=0$, we get that $|T|^2$ is constant and that $\langle
H,N\rangle=0$. We also have
$$
0=X(\langle H,
N\rangle)=\langle H,\nabla^{\perp}_XN\rangle=-|H|\langle E_{m+1},\sigma(T,X)\rangle
=-|H|\langle A_{m+1}T,X\rangle,
$$
for any tangent vector field $X$. Therefore, from Proposition
\ref{pTgen}, it follows that
$$
0=c(m-1)|T|^2(1-|T|^2),
$$
i.e. either $T=0$ or $T=\pm\xi$.

If $T=\pm\xi$, then $\Sigma^m$ is a vertical cylinder
$\pi^{-1}(\Sigma^{m-1})$ over a pmc submanifold $\Sigma^{m-1}$ in
$M^n(c)$ with second fundamental form $\sigma_0$, satisfying
$|\sigma_0|=|\sigma|$, and mean curvature vector field
$H_0=\frac{m}{m-1}H$. Then, condition \eqref{cond_thm} can be
rewritten as
$$
|\sigma_0|^2\leq (m-1)|H_0|^2-\frac{2c(m+1)}{m}<(m-1)|H_0|^2,
$$
which is a contradiction, since the squared norm of the traceless
part $\phi_0$ of $\sigma_0$ satisfies
$$
0\leq|\phi_0|^2=|\sigma_0|^2-(m-1)|H_0|^2.
$$

Hence, we have $T=0$, i.e. $\xi$ is normal to $\Sigma^m$. Since
$A_{\alpha}=0$ for all $\alpha>m+1$, it follows that the subbundle
$L=\Span\{\sigma\}=\Span\{H\}$ of the normal bundle is parallel,
i.e. $\nabla^{\perp}V\in L$ for all $V\in L$. Now, one can see that
$T\Sigma^m\oplus L$ is parallel, orthogonal to $\xi$, and invariant
by the curvature tensor $\bar R$. Using \cite[Theorem~2]{ET}, all
these lead to the conclusion that $\Sigma^m$ lies in an
$m+1$-dimensional totally geodesic submanifold of
$M^n(c)\times\mathbb{R}$, which is also orthogonal to $\xi$, i.e.
$\Sigma^m$ is a cmc hypersurface in $M^{m+1}(c)$.

\textbf{Case II: $m=3$.} We shall prove that $|\mathcal{A}|^2=0$ in
this situation too, which means, as we have seen above, that
$\Sigma^3$ is a cmc hypersurface in $M^4(c)$.

Our hypothesis \eqref{cond_thm} implies that the sequence
$\{\sigma_{ij}^{\alpha}(p_k)\}_{k\in\mathbb{N}}$, where
$$
\sigma_{ij}^{\alpha}=\langle\sigma(E_i,E_j),E_{\alpha}\rangle,
$$
is bounded for all $i$, $j$ and $\alpha$. We also know that the
sequence $\{|T|^2(p_k)\}_{k\in\mathbb{N}}$ is bounded. Therefore,
there exits a subsequence $\{p_{k_r}\}_{k_r\in\mathbb{N}}$ of
$\{p_k\}_{k\in\mathbb{N}}$ such that the following limits exit
$$
\bar\sigma_{ij}^{\alpha}=\lim_{k_r\rightarrow\infty}\sigma_{ij}^{\alpha}(p_{k_r})<\infty\quad\textnormal{and}\quad
|\bar T|^2=\lim_{k_r\rightarrow\infty}|T|^2(p_{k_r})<\infty,
$$
and we denote by
$$
\bar A_{\alpha}=\lim_{k_r\rightarrow\infty}A_{\alpha}(p_{k_r})
$$
the matrix with the entries $\bar\sigma_{ij}^{\alpha}$.

From $\lim_{k_r\rightarrow\infty}\Delta|\mathcal{A}|^2(p_{k_r})=0$,
it follows that, when we take the limit after $k_r\rightarrow\infty$,
all inequalities \eqref{estimate1}, \eqref{estimate2},
\eqref{estimate3}, and \eqref{estimate4} become equalities. Then,
from \eqref{estimate3} and \eqref{estimate4} we obtain
\begin{equation}\label{estimate3lim}
\sum_{\alpha,\beta>4}\{\trace[\bar A_{\alpha},\bar A_{\beta}]^2-(\trace(\bar A_{\alpha}\bar A_{\beta}))^2\}
=-\frac{3}{2}\Big(\sum_{\alpha>4}|\bar A_{\alpha}|^2\Big)^2=-\frac{3}{2}\Big(\sup_{\Sigma^3}|\mathcal A|^2\Big)^2
\end{equation}
and
\begin{equation}\label{estimate4lim}
\begin{array}{r}\sum_{\alpha>4}\{(\trace \bar A_{4})(\trace(\bar A_{\alpha}^2\bar A_{4}))-(\trace(\bar A_{\alpha}\bar A_{4}))^2\}
\\ \\=(\frac{3}{2}\sum_{\alpha>4}|\bar A_{\alpha}|^2+7c|\bar T|^2-3c)\sum_{\alpha>4}|\bar A_{\alpha}|^2\\ \\
=(\frac{3}{2}\sup_{\Sigma^3}|\mathcal A|^2+7c|\bar T|^2-3c)\sup_{\Sigma^3}|\mathcal A|^2,
\end{array}
\end{equation}
respectively. From \eqref{estimate3lim} and Lemma \ref{lili}, it
follows that either
\begin{enumerate}
\item  $\bar A_{5}=\ldots=\bar A_{n+1}=0$$;$ or

\item only two matrices $\bar A_{\alpha_0}$ and $\bar A_{\beta_0}$ are different from the null $m\times m$ matrix,
$|\bar A_{\alpha_0}|^2=|\bar A_{\beta_0}|^2=L$, and there exists
an orthogonal matrix $T$ such that
\begin{equation}\label{eq:matrix}
T^t\bar A_{\alpha_0}T=\sqrt{\frac{L}{2}}\left(\begin{array}{ccc}1&0&0\\0&-1&0\\0&0&0\end{array}\right),
\quad T^t\bar A_{\beta_0}T=\sqrt{\frac{L}{2}}\left(\begin{array}{ccc}0&1&0\\1&0&0\\0&0&0\end{array}\right).
\end{equation}
\end{enumerate}

In the first case, one obtains
$$
0=\sum_{\alpha>4}|\bar A_{\alpha}|^2=\sup_{\Sigma^3}|\mathcal A|^2,
$$
which means that $|\mathcal A|^2=0$ or, equivalently, that
$A_{\alpha}=0$ for all $\alpha>4$.

In the following, we shall assume that the second case occurs, and
we will come to a contradiction.

Restricting \eqref{eq:intermediar_e4} to the sequence of points
$\{p_{k_r}\}_{k_r\in\mathbb{N}}$ and then taking the limit, we get, also using
\eqref{estimate4lim}, that
$$
\sum_{i,j=1}^3\bar\lambda_i\bar\lambda_j(\bar\lambda_i^{\alpha}-\bar\lambda_j^{\alpha})^2
=\Big(\frac{1}{2}\sup_{\Sigma^3}|\mathcal{A}|^2+\frac{7c}{3}|\bar T|^2-c\Big)\sum_{i,j=1}^3(\bar\lambda_i^{\alpha}-\bar\lambda_j^{\alpha})^2,
$$
where $\bar\lambda_i=\lim_{k_r\rightarrow\infty}\lambda_i$ and
$\bar\lambda_i^{\alpha}=\lim_{k_r\rightarrow\infty}\lambda_i^{\alpha}$.
From \eqref{eq:matrix} we have
$\bar\lambda_i^{\alpha}\neq\bar\lambda_j^{\alpha}$ for $i\neq j$,
and then, from \eqref{eq:ij}, one obtains
\begin{equation}\label{eq:ijlim}
\bar\lambda_i\bar\lambda_j=\frac{1}{2}\sup_{\Sigma^3}|\mathcal{A}|^2+\frac{7c}{3}|\bar
T|^2-c \quad\textnormal{for}\quad i\neq j.
\end{equation}

Now, on the one hand, taking the limit in \eqref{eq:chen0} and
applying Lemma \ref{l:chen}, we get
$$
\Big(\sum_{i=1}^3\bar\lambda_i\Big)^2=2\sum_{i=1}^3(\bar\lambda_i)^2
+2\Big(\sup_{\Sigma^3}|\mathcal{A}|^2+\frac{14c}{3}|\bar T|^2-2c\Big),
$$
or, equivalently,
\begin{equation}\label{eq:final1}
\frac{3}{2}|H|^2=|\bar\phi_{4}|^2+\sup_{\Sigma^3}|\mathcal{A}|^2+\frac{14c}{3}|\bar T|^2-2c,
\end{equation}
where $\phi_4=A_4-|H|\id$ is the traceless part of $A_4$ and
$\bar\phi_4=\lim_{k_{r}\rightarrow\infty}\phi_4(p_{k_r})$.

On the other hand, we have
$$
\sum_{i\neq j}\lambda_i\lambda_j=\Big(\sum_{i=1}^3\lambda_i\Big)^2-\sum_{i=1}^3(\lambda_i)^2
=9|H|^2-(|\phi_4|^2+3|H|^2)=6|H|^2-|\phi_4|^2,
$$
which, tacking the limit and using \eqref{eq:ijlim}, gives
\begin{equation}\label{eq:final2}
|\bar\phi_4|^2=6|H|^2-3\sup_{\Sigma^3}|\mathcal{A}|^2-14c|\bar T|^2+6c.
\end{equation}

Summarizing, from \eqref{eq:final1} and \eqref{eq:final2}, one
obtains
$$
|\bar\phi_4|^2=-\frac{3}{4}|H|^2,
$$
which is a contradiction and, therefore, this case cannot occur.

We have just proved that our submanifold $\Sigma^m$ actually is a
cmc hypersurface in $M^{m+1}(c)$ for any $m\geq 3$.

Now, from \eqref{cond_thm}, it is easy to see that
$$
|\phi|^2\leq 2c+\frac{m}{m-1}|H|^2<r^2,
$$
where $\phi$ is the traceless part of $\sigma$ and $r$ is the
positive root of the polynomial
$$
P(t)=t^2+\frac{m(m-2)}{\sqrt{m(m-1)}}|H|t-m(c+|H|^2).
$$
We then use \cite[Theorem~1.5]{AdC} (see also \cite{WS}) to conclude
that $\phi=0$, i.e. $\Sigma^m$ is a totally umbilical cmc
hypersurface in $M^{m+1}(c)$.
\end{proof}

\begin{theorem}\label{thm:pmc_red_2} Let $\Sigma^m$ be a complete non-minimal pmc submanifold in
$M^n(c)\times\mathbb{R}$, $n>m\geq 3$, $c<0$, with mean curvature vector field $H$ and second fundamental form $\sigma$. If $H$ is orthogonal to $\xi$  and
\begin{equation}\label{cond_thm_2}
|\sigma|^2+\frac{2c(m+1)}{m}|T|^2\leq 4c+\frac{m^2}{m-1}|H|^2,
\end{equation}
then $\Sigma^m$ is a totally umbilical cmc hypersurface in
$M^{m+1}(c)$.
\end{theorem}

\begin{proof} Let us consider a local orthonormal frame field $\{E_{m+1},\ldots,E_{n+1}\}$ in the normal bundle such that $E_{m+1}=\frac{H}{|H|}$. Then, since $H\perp\xi$, we have
$$
A_N=\sum_{\alpha>m+1}\langle N,E_{\alpha}\rangle A_{\alpha}
$$
and, therefore, from the Schwarz inequality, one obtains
$$
\begin{array}{ll}
|A_N|^2&=\Big|\sum_{\alpha>m+1}\langle N,E_{\alpha}\rangle A_{\alpha}\Big|^2\leq\Big(\sum_{\alpha>m+1}|\langle N,E_{\alpha}\rangle| |A_{\alpha}|\Big)^2\\ \\&\leq\Big(\sum_{\alpha>m+1}|\langle N,E_{\alpha}\rangle|^2\Big)\Big(\sum_{\alpha>m+1}|A_{\alpha}|^2\Big)\leq |N|^2|\mathcal A|^2\\ \\&=(1-|T|^2)|\mathcal A|^2,
\end{array}
$$
where $|\mathcal A|^2=\sum_{\alpha>m+1}|A_{\alpha}|^2$. Then, from \eqref{eq:A},
it follows that
\begin{equation}\label{eq:AHR}
\begin{array}{lcl}
\frac{1}{2}\Delta|\mathcal{A}|^2&\geq&c(2m-(m+1)|T|^2)|\mathcal{A}|^2\\ \\&&+\sum_{\alpha>m+1}\{(\trace A_{m+1})(\trace(A_{\alpha}^2A_{m+1}))-(\trace(A_{\alpha}A_{m+1}))^2\}\\ \\&&+\sum_{\alpha,\beta>m+1}\{\trace[A_{\alpha},A_{\beta}]^2-(\trace(A_{\alpha}A_{\beta}))^2\},
\end{array}
\end{equation}
where we also used the fact that $-c\sum_{\alpha>m+1}|A_{\alpha} T|^2\geq 0$.

Next, in the same way as in the proof of Theorem \ref{thm:pmc_red_1}, we get
$$
\sum_{\alpha,\beta>m+1}\{\trace[A_{\alpha},A_{\beta}]^2-(\trace(A_{\alpha}A_{\beta}))^2\}\geq
-\frac{3}{2}|\mathcal{A}|^4
$$
and, using \eqref{cond_thm_2},
$$
\begin{array}{r}
\sum_{\alpha>m+1}\{(\trace A_{m+1})(\trace(A_{\alpha}^2A_{m+1}))-(\trace(A_{\alpha}A_{m+1}))^2\}\\ \\\geq (\frac{m}{2}|\mathcal{A}|^2+c(m+1)|T|^2-2cm)|\mathcal{A}|^2.
\end{array}
$$

Replacing in \eqref{eq:AHR}, we obtain that
$$
\frac{1}{2}\Delta|\mathcal{A}|^2\geq\frac{m-3}{2}|\mathcal{A}|^4,
$$
which, again as in the proof of Theorem \ref{thm:pmc_red_1}, implies that $|\mathcal A|^2=0$, and, therefore, $A_{\alpha}=0$ for all $\alpha>m+1$.

On the other hand, since $H\perp\xi$ implies that $A_{m+1}T=0$, and $A_N=0$ implies that $|T|=\cst$, from Proposition \ref{pTgen}, we can see that
$$
0=c(m-1)|T|^2(1-|T|^2),
$$
which means that either $T=0$ or $T=\pm\xi$. If $T=\pm\xi$, then $\Sigma^m$ is a vertical cylinder
$\pi^{-1}(\Sigma^{m-1})$ over a pmc submanifold $\Sigma^{m-1}$ in
$M^n(c)$, with second fundamental form $\sigma_0$, satisfying
$|\sigma_0|=|\sigma|$, and mean curvature vector field
$H_0=\frac{m}{m-1}H$. Then, from \eqref{cond_thm_2}, it follows that
$$
|\sigma_0|^2\leq (m-1)|H_0|^2+\frac{2c(m-1)}{m}<(m-1)|H_0|^2,
$$
which is a contradiction. Hence $T=0$ and, using \cite[Theorem~2]{ET}, this leads to the conclusion that $\Sigma^m$ is a cmc hypersurface in $M^{m+1}(c)$.

Finally, we observe that, using \eqref{cond_thm_2}, we have
$$
|\phi|^2\leq 4c+\frac{m}{m-1}|H|^2<r^2,
$$
where $\phi$ is the traceless part of $\sigma$ and $r$ is the
positive root of the polynomial
$$
P(t)=t^2+\frac{m(m-2)}{\sqrt{m(m-1)}}|H|t-m(c+|H|^2),
$$
and then, from \cite[Theorem~5]{AGM}, we get that $\phi=0$, which
means that $\Sigma^m$ is totally umbilical in $M^{m+1}(c)$.
\end{proof}

In the case of pmc surfaces, we can state the following two results.

\begin{theorem}\label{thm:surface_red1} Let $\Sigma^2$ be a complete non-minimal pmc surface in $M^n(c)\times\mathbb{R}$, $n>2$, $c>0$, such that
the angle between $H$ and $\xi$ is constant and
$$
|\sigma|^2+3c|T|^2\leq 4|H|^2+2c.
$$
Then, either
\begin{enumerate}
\item $\Sigma^2$ is pseudo-umbilical and lies in $M^n(c)$$;$ or

\item $\Sigma^2$ is a torus $\mathbb{S}^1(r)\times\mathbb{S}^1\Big(\sqrt{\frac{1}{c}-r^2}\Big)$ in $M^3(c)$, with $r^2\neq\frac{1}{2c}$.
\end{enumerate}
\end{theorem}

\begin{proof} The map
$p\in\Sigma^2\rightarrow(A_H-\mu\id)(p)$, where $\mu$ is a constant, is analytic, and,
therefore, either $\Sigma^2$ is a pseudo-umbilical surface (at every
point), or $H$ is an umbilical direction on a closed set without interior
points. In the second case, $H$ is not an umbilical direction on an open dense set $W$. We shall work on this set and then we shall extend the results to the whole surface by continuity.

If $\Sigma^2$ is a pmc surface in $M^n(c)\times\mathbb{R}$, then either $\Sigma^2$ is
pseudo-umbilical, i.e. $H$ is an umbilical direction everywhere, or, at any point in $W$, there exists a local orthonormal frame field that diagonalizes $A_U$ for
any normal vector field $U$ defined on $W$ (see \cite[Lemma 1]{AdCT}). According to \cite[Theorem~1]{AdCT}, if $\Sigma^2$ is a pseudo-umbilical pmc surface in $\mathbb{S}^n(c)\times\mathbb{R}$, then it lies in $M^n(c)$, and if the surface is not pseudo-umbilical, then it lies in $M^4(c)\times\mathbb{R}$.

In the following, we shall assume that $\Sigma^2$ is not pseudo-umbilical and we shall prove that, in this case, it is a torus in $M^3(c)$.

First, let $\{E_3=\frac{H}{|H|},E_4,E_5\}$ be a local orthonormal frame field in the normal bundle, and then observe that $[A_{\alpha},A_{\beta}]=0$ for all $\alpha$ and $\beta$. Moreover, diagonalizing simultaneously $A_4$ and $A_5$, we easily get
$$
(\trace(A_4A_5))^2=2|A_4|^2|A_5|^2\leq\frac{1}{2}(|A_4|^2+|A_5|^2)^2=\frac{1}{2}|\mathcal A|^4,
$$
which means that
\begin{equation}\label{estimate_surface}
\trace[A_4,A_5]^2-(\trace(A_4A_5))^2=-2|A_4|^2|A_5|^2\geq-\frac{1}{2}|\mathcal A|^4.
\end{equation}

Now, taking into account that
$$
|A_{\alpha}T|^2=\frac{1}{2}|T|^2|A_{\alpha}|^2
$$
for $\alpha\in\{4,5\}$, since $\trace A_{\alpha}=0$, and then working exactly as in the proof of Theorem \ref{thm:pmc_red_1}, we obtain
$$
\Delta|\mathcal A|^2\geq\frac{1}{2}|\mathcal A|^4\geq 0.
$$

By hypothesis, we have that the Gaussian curvature $K$ of our surface satisfies
$$
0=2K=2c(1-|T|^2)+4|H|^2-|\sigma|^2\geq c|T|^2\geq 0,
$$
which means that $\Sigma^2$ is a parabolic space. Therefore, since $|\mathcal A|^2$ is a bounded subharmonic function, we get that $|\mathcal A|^2=0$, i.e. $A_4=A_5=0$. Moreover, using Proposition \ref{pTgen}, we can see that either $T=0$ or $T=\pm\xi$. Again as in Theorem \ref{thm:pmc_red_1} we discard the second case and we conclude that $\Sigma^2$ lies in $M^3(c)$ by using \cite[Theorem~2]{ET}.

Finally, since $\Sigma^2$ is not pseudo-umbilical, from a result in \cite{DAH} (see also \cite[Theorem~1.5]{AdC}), we obtain that $|\sigma|^2=4|H|^2+2c$ and that our surface is the torus $\mathbb{S}^1(r)\times\mathbb{S}^1\Big(\sqrt{\frac{1}{c}-r^2}\Big)$, with $r^2\neq\frac{1}{2c}$.
\end{proof}

\begin{theorem}\label{thm:surface_red2} Let $\Sigma^2$ be a complete non-minimal pmc surface in $M^n(c)\times\mathbb{R}$, $n>2$, $c<0$, such that
$H$ is orthogonal to $\xi$ and
$$
|\sigma|^2+5c|T|^2\leq 4|H|^2+4c.
$$
Then $\Sigma^2$ is pseudo-umbilical and lies in $M^n(c)$.
\end{theorem}

\begin{proof} Let us assume that $\Sigma^2$ is not pseudo-umbilical. Then, from \eqref{estimate_surface}, and working as in Theorem \ref{thm:pmc_red_2} we can prove that $\Sigma^2$ lies in $M^3(c)$. On the other hand, we observe that $|\sigma|^2\leq 4|H|^2+4c<4|H|^2+2c$, and, therefore, using a result in \cite{T}, we have that the surface is totally umbilical, which is a contradiction.
\end{proof}

\end{document}